\documentclass[10pt]{article}
\font\smallit=cmti10

\usepackage{amssymb,latexsym,amsmath,epsfig,amsthm} 

\usepackage{numprint}

\makeatletter

\renewcommand\section{\@startsection {section}{1}{\z@}
{-30pt \@plus -1ex \@minus -.2ex}
{2.3ex \@plus.2ex}
{\normalfont\normalsize\bfseries\boldmath}}

\renewcommand\subsection{\@startsection{subsection}{2}{\z@}
{-3.25ex\@plus -1ex \@minus -.2ex}
{1.5ex \@plus .2ex}
{\normalfont\normalsize\bfseries\boldmath}}

\renewcommand{\@seccntformat}[1]{\csname the#1\endcsname. }

\makeatother

\newtheorem{theorem}{Theorem}
\newtheorem{lemma}{Lemma}
\newtheorem{conjecture}{Conjecture}

\newtheorem{corollary}{Corollary}

\newcommand{\ZL}{\mathcal{R}(L)}

\newcommand{\Z}{\mathbb{Z}}

\newcommand{\R}{\mathbb{R}}
\newcommand{\Q}{\mathbb{Q}}
\newcommand{\PL}{{\mathcal{G}\mathcal{P}}(L)}
\newcommand{\DL}{{\mathcal{D}}(L)}
\newcommand{\al}{\alpha}
\newcommand{\ao}{\alpha_0}
\newcommand{\an}{\alpha_n}
\newcommand{\zi}{\Z[i]}
\newcommand{\re}{{\text{Re}}}
\newcommand{\im}{{\text{Im}}}
\newcommand{\be}{\begin{equation}}
\newcommand{\ee}{\end{equation}}
\newcommand{\beq}{\begin{eqnarray*}}
\newcommand{\eeq}{\end{eqnarray*}}

\begin{document}

\begin{center}
\uppercase{\bf Walking to infinity along Gaussian lines}
\vskip 20pt
{\bf Elsa Magness
}\\
{\smallit Department of Mathematics, Seattle University, Seattle, WA 98122, USA}\\
{\tt emagness@brynmawr.edu}\\
\vskip 10pt
{\bf Brian Nugent
}\\
{\smallit Department of Mathematics, Seattle University, Seattle, WA 98122, USA}\\
{\tt bnugent@uw.edu}\\ 
\vskip 10pt
{\bf Leanne Robertson}\\
{\smallit Department of Mathematics, Seattle University, Seattle, WA 98122, USA}\\
{\tt robertle@seattleu.edu}\\ 
\end{center}
\vskip 20pt
\vskip 30pt

\centerline{\bf Abstract}
\noindent
We study analogies between the rational integers on the real line and the Gaussian integers on other lines in the complex plane. This includes a Gaussian analog of Bertrand's Postulate, the Chinese Remainder Theorem, and the periodicity of divisibility. We also computationally investigate the distribution of Gaussian primes along these lines and leave the reader with several open problems. 

\pagestyle{myheadings}
\thispagestyle{empty}
\baselineskip=12.875pt
\vskip 30pt

\section{Introduction.}\label{intro}
Is it possible to walk from the origin in the complex plane to infinity using steps of bounded length and stepping only on Gaussian primes? Several authors have worked on this intriguing question since it was first posed by Basil Gordon  in 1962.  Erd\"os conjectured that such a walk to infinity is  impossible, but the problem remains unsolved today (see \cite{Gethner} for a discussion of the contradictory references to Erd\"os' role in this problem). In 1970, Jordan and Rabung \cite{JR} showed that steps of length at least 4 would be required, and in 1998, Gethner, Wagon, and Wick~\cite{Gethner} showed that steps of length $\sqrt{26}$ or less are insufficient to reach infinity. In the same paper they showed that it is impossible to walk to  infinity on any line in the complex plane by stepping only on Gaussian primes and taking steps of bounded length, and thus established the Gaussian analog of the classical result that there are arbitrarily long sequences of composites on the real line. In 2017, West and Sittinger \cite{West} generalized 
this result and showed that in any quadratic field of class number 1, it is similarly impossible  to walk to infinity along any line using steps of bounded length and stepping only on primes in the ring of integers of the field. Motivated by these results, we further investigate the idea of walking to infinity on lines in the complex plane stepping only on Gaussian integers,  and  analogies to walking to infinity along the real line.

Recall that the ring $\Z [i]$ of {\em Gaussian integers} consists of all complex numbers of the form $\alpha = a+bi$, where $a$ and $b$ are rational integers.  Following Gethner et al., we call a line  in the complex plane  a {\em{Gaussian line}} if it contains two, and hence infinitely many, Gaussian integers. We call a Gaussian line {\em{primitive}} if the integers on the line do not all share a common divisor.  With these definitions, we ask what you might discover if instead of wandering freely on Gaussian integers in the complex plane,  you walked along a primitive Gaussian line stepping only on Gaussian integers?  How different or similar would this stroll to infinity be to that of walking  to infinity along the real line stepping only on rational integers?  Would you stroll on infinitely many Gaussian primes, or perhaps none at all? Could you observe an analog of Bertrand's postulate on your walk? Would you see a  periodicity of divisibility similar to that on the real line? What other properties of the Gaussian integers might you discover?

An overview of the paper and our results is as follows.  In Section \ref{background}, we establish the background and notation used throughout.
 In Section \ref{primes}, we investigate the distribution of Gaussian primes on Gaussian lines.  We discuss what a theorem of T. Tao says about primes on Gaussian lines and formulate and computationally support an extension of  Bertrand's Postulate to these lines. The main questions posed in this section  are equivalent to famous open problems about quadratic polynomials representing prime numbers, so we turn to more tractable problems in subsequent sections. In Section~\ref{div}, we prove key divisibility properties of Gaussian integers on Gaussian lines that are important for the rest of the paper.   This includes an analogy of the periodicity of divisibility of rational integers on the real line and a characterization of  the rational  integers and Gaussian primes  that divide some Gaussian integer on a given Gaussian line.  In Section~\ref{CRT}, we extend the Chinese Remainder Theorem  to Gaussian lines and prove a theorem that shows there are always infinitely many Gaussian lines that satisfy any given CRT-type divisibility properties.  
 Finally, in Section \ref{divisor} we return to questions raised in Section \ref{div} and completely characterize the set of Gaussian integers that divide some Gaussian integer on a given Gaussian line.

\section{Background and Notation.}\label{background}

We begin with some background on Gaussian integers and by establishing the notation concerning Gaussian lines that is used throughout the paper.   

\medskip

The unit group of the Gaussian integers $\Z[i]$ is
$\{\pm 1, \pm i\}$, so  two Gaussian integers, $\alpha$ and $\beta$, are {\em associates} if and only if $\alpha=\pm \beta$ or $\alpha = \pm i \beta$.
The {\em norm} of the Gaussian integer $\alpha=a+bi$ is defined by $N(a+bi)=\alpha\cdot\overline\alpha =a^2+b^2\in\Z$, where the ``bar'' denotes complex conjugation, and its {\em trace} is defined by $Tr(a+bi)=\alpha+\overline\alpha=2a\in\Z$.  Unique factorization holds in $\Z[i]$, and this gives the Gaussian integers a well-defined notion of primality. To avoid confusion, we use the terminology {\em rational prime} for a prime in the rational integers $\Z$, and {\em Gaussian prime} for a prime in $\Z [i]$. 

The Gaussian primes can be classified in terms of the factorization of the rational primes $p\in \Z$ into Gaussian primes as follows:

     \begin{center}
     \begin{enumerate}
	 \item If $p=2$, then $p$ ramifies in $\Z [i]$. Specifically, $2=-i(1+i)^2$, so $1+i$ is a Gaussian prime of norm 2.
          \item If $p \equiv 1 \pmod 4$, then $p = \pi\cdot\overline{\pi}$ splits as a product of two conjugate Gaussian primes of norm $p$ that are not associates in $\Z[i]$.
          \item If $p \equiv 3 \pmod 4$, then $p$ remains prime in $\Z [i]$ and has norm $p^2$.\\
     \end{enumerate}
     \end{center}
Every Gaussian prime is an associate of one of the Gaussian primes described above.
If $\pi$ is a Gaussian prime then we say $\pi$ \textit{lies over} $p$ if $\pi$ divides the rational prime $p$.

For every Gaussian line $L$, we distinguish two Gaussian integers, $\ao=a+bi$   and $\delta=c+di$, that define $L$ as follows. Let $\ao$ be the Gaussian integer on $L$ of minimum norm, and if there are two such integers, let $\ao $ be the one with the larger real part.  If $L$ is vertical, then take $\delta=i$.  Otherwise, let $\alpha_1$ be the Gaussian integer on $L$ closest to $\ao$ (so $N(\alpha_1-\ao)$ is minimal) and with $\re(\alpha_1)>\re(\ao)$.  Then take $\delta=\alpha_1-\ao$.  Thus $\alpha_0$ is on the line $L$, but  $\delta$ is not, provided $\ao\neq 0$.  Note that there are only two primitive Gaussian lines with $\ao=0$, namely the real line $\im(z)=0$ and the imaginary line $\re(z)=0$.

With $\ao$ and $\delta$  defined in this way, the  lemma below describes all Gaussian integers on $L$.  This lemma is essentially  Lemma 4.2 in \cite{Gethner}, except that we describe the primitive case and specify $\ao$ and $\delta$,  since this is convenient for our work.

\begin{lemma}\label{notation}  Let $L$ be a Gaussian line, and let $\ao=a+bi$ and $\delta=c+di$ be as defined above.  Then $c$ and $d$  are  relatively prime, $c\geq 0$, and  the Gaussian integers  on $L$ are exactly the Gaussian integers $\an$ given by
$$ \an=\ao+\delta n,\  n\in\Z.\label{an}$$
Moreover, $L$ is primitive if and only if $\ao$ and $\delta$ are relatively prime over $\Z[i]$.
\end{lemma}

\begin{proof}  If $L$ is vertical then $\delta=i$ and  $\ao=k$ for some $k\in\Z$.  Then the Gaussian integers on $L$ are given by $\an=k+ni, n\in\Z$, $L$ is primitive, and  $\ao$ and $\delta$ are relatively prime.  Thus, the lemma holds for all vertical Gaussian lines.

If  $L$ is not vertical, then by our choice of $\delta=c+di$ we have $c> 0$  and $L$ has slope $d/c$. Thus $c$ and $d$ must be relatively prime since otherwise there would be a Gaussian integer on $L$ between $\ao$ and $\alpha_1$, contradicting our choice of $\alpha_1$.  Let $\beta$ be a Gaussian integer on $L$. Then $\beta=\ao+r\delta$ for some real number $r$.  But, $r=(\beta-\ao)/\delta$ is in the quotient field $\Q(i)$, so $r\in\Q$.  Now $r\delta=rc+rdi=\beta-\ao\in\Z[i]$ implies $rc, rd\in\Z$.  Since $c$ and $d$ are relatively prime, it follows that $r\in\Z$ as needed.

For the second part of the lemma, first suppose $\ao$ and $\delta$  have a common Gaussian prime divisor $\pi$.  Then $\pi$ divides $\ao+\delta n$ for all $n\in\Z$, $i.e.$, $\pi$ divides all Gaussian integers $\alpha_n$ on $L$ and $L$ is not primitive.  Conversely, if $\ao$ and $\delta$ are relatively prime, then $\ao$ and ${\alpha_1}=\ao+\delta$ are also relative prime, and $L$ must be primitive since it contains at least two Gaussian integers that do not share a common divisor. \end{proof}

Throughout this paper, we define a Gaussian line $L$ by its values of 
$\alpha_0$  and $\delta$ as given in Lemma \ref{notation}. Given these values, we also define a rational integer $\Delta$ associated to $L$ by 
\begin{equation}\label{D}\Delta=ad-bc.\end{equation}
Note that if $\alpha_n=x+yi=\ao+n\delta$, $n\in\Z$, is any other Gaussian integer on $L$, then  $x=a+nc$ and $y=b+nd$ and so
$\Delta$ can also be computed by
$\Delta=xd-yc$. That is, $\Delta$ can easily be computed from the values of $\alpha_n$ and $\delta$, not just from $\ao$ and $\delta$.  In Section \ref{div}, we use $\Delta$ to characterize the set of rational integers that divide some Gaussian integer on $L$. 
Another use of $\Delta$ is given by the following easy lemma.

\begin{lemma}\label{delta} Let  $L$ be a primitive Gaussian line. Then $\Delta=0$ if and only if $L$ is the real or imaginary line, which holds if and only if $\alpha_0=0$.\end{lemma}

\begin{proof} The only part of the lemma that doesn't follow directly from the definitions is the fact that if $\Delta=0$ then $L$ is the real or imaginary line. For this assume $\Delta=0$, so
 $ad=bc$.  Since $c$ and $d$  are relatively prime, it follows that $c\mid a$ and $d\mid b$.  Thus, $a=cx$ and $b=dy$ for some $x,y\in\Z$.  This gives $cdx=cdy$.  We may assume $c$ and $d$ are both nonzero since otherwise $a$ or $b$ is equal to zero and $L$ is the real or imaginary line.  Thus, it follows that $x=y$ and $\ao=x\delta$.  Hence, $x=0$ since $\ao$ and 
$\delta$ are relatively prime and $\delta\neq 0$.  Therefore, $\ao=0$, and $L$ is either the real or imaginary line.
\end{proof}

\section{Primes on Gaussian Lines.}\label{primes} One of the first questions we had when we began our study of Gaussian lines was about the distribution of Gaussian primes  on these lines.  We wondered whether every primitive Gaussian line contains infinitely many  Gaussian primes, or if the existence of even one prime is guaranteed.   This led us to consider what T.~Tao's~\cite{Tao} beautiful theorem about arbitrarily shaped constellations in the Gaussian primes says about primes on Gaussian lines, and to formulate and computationally support an analog of Bertrand's Postulate to Gaussian lines. 
\medskip

The real and imaginary lines contain infinitely many primes, so it is natural to wonder whether every primitive Gaussian line similarly contains infinitely many Gaussian primes.  Finding even one other primitive Gaussian line that contains infinitely many Gaussian primes is a  very difficult problem, however, since this is equivalent (by taking norms) to finding a quadratic polynomial that takes on infinitely many rational prime values and no such polynomials are known. For example, determining whether or not there are infinitely many Gaussian primes on the Gaussian line with 
$\alpha_0=1$ and $\delta=i$ ($i.e.$ Gaussian primes of the form $\an=1+ni$) is equivalent to determining whether or not there are infinitely many rational primes of the form $1+n^2$, which is
Landau's fourth problem given at the 1912 International Congress of Mathematicians and remains open today.   In general, it is also not known whether every irreducible quadratic polynomial attains at least one prime value, so similarly we cannot easily decide whether every primitive Gaussian line contains at least one Gaussian prime.

%

Despite the difficulty of finding a Gaussian line that contains infinitely many Gaussian primes, we can apply a result of Iwaniec and Lemke Oliver  to prove that infinitely many Gaussian lines contain infinitely many elements that are the product of at most two Gaussian primes.  
 For example,  it is a deep theorem of Iwaniec \cite{Iwaniec} 
 that there are infinitely many values of $n$ such that $1+n^2$ is the product of at most two rational primes, from which it is immediate that  the vertical Gaussian line defined by $\alpha_0=1$ and $\delta =i$ contains infinitely many elements that are the product of at most two Gaussian primes.   Iwaniec notes that his proof generalizes to show that 
  if $G(n)=An^2+Bn+C$ is an irreducible polynomial with $A>0$ and $C$ odd, then there exist infinitely many integers $n$ such that $G(n)$ has at most two rational prime factors. This theorem also follows from a result of Lemke Oliver~\cite{Oliver}  generalizing Iwaniec's work. Applied to Gaussian lines, this result yields the following:

\begin{theorem}\label{2primes}
Let $L$ is a primitive Gaussian line such that $1+i $ does not divide $\ao$.  Then $L$ contains infinitely Gaussian integers that are the product of at most two Gaussian primes.  
\end{theorem}

\begin{proof} Let $L$ be a primitive Gaussian line with $\ao=a+bi$, $\delta=c+di$, and $\Delta=ad-bc$ as defined in (\ref{D}). Assume $1+i$ does not divide $\ao$.  
The norm of an arbitrary Gaussian integer $\alpha_n$ on $L$ can be viewed as a quadratic polynomial $f(n)$ as follows:
\begin{align}\label{N} f(n)=N(\alpha_n)=N(\alpha_0+\delta n)
&=N(\delta)n^2+Tr(\alpha_0\overline\delta)n+N(\alpha_0)\\
&=(c^2+d^2)n^2+2(ac+bd)n+a^2+b^2\nonumber.
\end{align}
The discriminant  of $f$ is equal to $-4\Delta^2$, which is negative unless $\Delta=0$.    Thus, $f$ is irreducible over $\Z$ unless $L$ is the real or imaginary line, by Lemma \ref{delta}. Moreover, the leading coefficient  of $f$ is positive and the constant term $N(\alpha_0)$ is odd, since we are assuming $1+i$ does not divide $\ao$.  It follows from  Iwaniec's theorem discussed above that  there are infinitely many $n$ such that $f(n)=N(\alpha_n)$ is a product of at most two rational primes, $i.e.$,  $\alpha_n$ is a product of at most two Gaussian primes.
\end{proof}

Unfortunately Theorem \ref{2primes} does not say anything about the distribution of Gaussian primes on Gaussian lines.  For this, we first apply T. Tao's \cite{Tao} astonishing theorem about arbitrarily shaped constellations in the Gaussian primes to Gaussian lines.  His theorem says the following:

\begin{theorem}[T. Tao \cite{Tao}] Given any distinct Gaussian integers $v_1,\ldots,v_{k}$, there are infinitely many sets $\{\alpha+rv_1,\ldots, \alpha+rv_{k}\}$, with $\alpha\in\Z[i]$ and $r\in\Z\setminus \{0\}$, all of whose elements are Gaussian primes.\end{theorem}

By choosing $\delta=c+di\in\Z[i]$ with $\gcd(c,d)=1$ as usual, we can apply Tao's theorem with $v_1=\delta$, $v_2=2\delta,\ldots,v_{k}=k\delta$. The theorem guarantees the existence of infinitely many pairs $(\alpha, r)$ such that all the elements in the set  
$$P_{\alpha, r}=\{\alpha+r\delta,\alpha+2r\delta,\ldots, \alpha+kr\delta\}$$
are Gaussian primes. For each $\alpha$, there is a primitive Gaussian line $L_\alpha$ with slope $m=d/c$ ($i.e.$, $\delta=c+di$) that passes through all the elements in $P_{\alpha, r}$.  Thus, $L_\alpha$ contains $k$ Gaussian primes in arithmetic progression.  It is possible that infinitely many of the sets $P_{\alpha, r}$ are actually on the same Gaussian line (that is, infinitely many of the lines $L_\alpha$ have the same $\ao$).  In this case, we thus have a Gaussian line that contains infinitely many Gaussian primes.  It  follows that for a fixed slope $m\in\Q$, either there is a Gaussian line with slope $m$ that contains infinitely many Gaussian primes or, for all $k\geq 1$, there are infinitely many Gaussian lines with slope $m$ that contain $k$ Gaussian primes in arithmetic progression.  Considering this for all $m$  and excluding the real and imaginary lines (the case $\ao=0$),  gives the following:

\begin{corollary} At least one of the following two statements is true:
\begin{enumerate}
\item There is a Gaussian line with $\alpha_0\neq 0$ that contains infinitely many Gaussian primes.
\item For every rational integer $m$ and every positive integer $k$, there are infinitely many distinct Gaussian lines with slope $m$ that contain $k$ Gaussian primes in arithmetic progression.
\end{enumerate}
\end{corollary}

Note that if the first statement in the corollary is true, then by taking norms it is also true that there is a quadratic polynomial that takes on infinitely many prime values.
Regarding the second statement, note that it is not possible for a Gaussian line to contain infinitely many Gaussian primes in arithmetic progression.  This follows from the result of Gethner et al. \hskip-.05in \cite{Gethner}  mentioned earlier that every Gaussian line contains arbitrarily long sequences of consecutive Gaussian composites.

We also wondered where to look for primes on Gaussian lines.  On the real line,  Bertrand's Postulate  guarantees  the existence of  a rational prime between $n$ and $2n$ for every rational integer $n\geq 3$. In other words, there exists a prime between  $n$ and the next integer that is  divisible by $n$. We wondered if the analogous statement holds on Gaussian lines. If $\alpha_n$ is on a Gaussian line $L$ then to characterize the next Gaussian integer on $L$ divisible by $\alpha_n$, we define a function $\nu:\Z[i]\rightarrow \Z$ by 
\begin{equation}\label{morm}
\nu(x+iy) = \frac{N(x+iy)}{{\rm gcd}(x,y)}.
\end{equation}
The function $\nu$  is useful  because if $\beta\in\Z[i]$ then the smallest positive rational integer divisible $\beta$ is $\nu(\beta)$, and furthermore, $\nu(\beta)$ divides every rational integer that is divisible by $\beta$.  For example, if $\beta = 2+6i = 2(1+3i)$ then the smallest positive rational integer divisible by $\beta$ is $2(1+3i)(1-3i)=20=\nu(\beta)$,
and $\beta$ divides a rational integer $r$ if and only if $r$ is divisible by 20.
With regards to Bertrand's postulate, if $\alpha_n$ is on a Gaussian line $L$ then  the next Gaussian integer on $L$ divisible by $\alpha_n$ is 
$\alpha_{n+\nu(\alpha_n)}=\alpha_{n}+\nu(\alpha_n)\cdot\delta$. Notice that $\nu(r)=r$ for all $r\in \Z$, so Conjecture \ref{bertrand2} below is equivalent to Bertrand's Postulate when $L$ is the real line.  We include a second conjecture because
 $\alpha_{n+N(\alpha_n)}=\alpha_n+ N(\alpha_n)\cdot\delta$ is also divisible by $\alpha_n$ and, as we discuss below, it is more efficient to use the norm when searching for Gaussian primes on lines.
Thus, we make the following two conjectures that generalize Bertrand's Postulate.

\medskip
\begin{conjecture} [Strong Bertrand for Gaussian lines]  \label{bertrand2}
Let $L$ be a primitive Gaussian line.  If $n>1$,  then there is always at least one Gaussian prime on $L$ that lies between $\alpha_n$ and $\alpha_{n+\nu(\alpha_n)}$.
\end{conjecture}

\begin{conjecture} [Weak Bertrand for Gaussian lines] \label{bertrand}
Let $L$ be a primitive Gaussian line. If $n>1$,  then there is always at least one Gaussian prime on $L$ that lies between  $\alpha_n$ and $\alpha_{n+N(\alpha_n)}$.
\end{conjecture}

\medskip
We wrote a program in Sage \cite{Sage} to search for lines $L$ where Conjecture~\ref{bertrand}  fails for some Gaussian integer on $L$.  We tested well over $10^{10}$ consecutive Gaussian integers on about $\numprint{700,000}$ lines and the conjecture held in every case.  About $\numprint{607,000}$ of the lines we checked had $\ao=1$ and $\delta=c+di$, where $c$ and $d$ were relatively prime integers ranging from one to 1,000.  Additionally, we checked over 24,000 lines where $c$ and $d$  were random integers between 300 and  $10^{18}$. Finally, we checked about 65,000 lines with $\ao\neq 1$. 

Our algorithm for testing Conjecture \ref{bertrand} relies on the fact that if $\alpha_\ell=\pi$ is a Gaussian prime between $\alpha_n$ and $\alpha_{n+N(\alpha_n)}$ for some $0<n<\ell$, then $\pi$ is also between $\alpha_k$ and $\alpha_{k+N(\alpha_k)}$ whenever $n<k<\ell$.  This holds because $N(\alpha_n)<N(\alpha_k)$ whenever $0<n<k$ by our choice of $\alpha_0$ being the element of smallest norm on $L$. The corresponding statement does not hold for $\nu(\alpha_n)$, which is why we focus on Conjecture~\ref{bertrand}.
Specifically, for every line $L$ that we tested, we found a sequence of $10^{10}$ Gaussian integers $\alpha_{\ell_i}$, $1\leq i\leq 10^{10}$, on $L$ such that the following three conditions are satisfied:

\begin{enumerate}
\item Each $\alpha_{\ell_i}$ is a Gaussian prime;
\item The first Gaussian prime  in the sequence, $\alpha_{\ell_1}$, lies between $\alpha_1$ and $\alpha_{1+N(\alpha_1)}$, $i.e.$, $1<\ell_1<1+N(\alpha_1)$;
\item For $i\geq 1$, the Gaussian prime $\alpha_{\ell_{i+1}}$ lies between the previous prime $\alpha_{\ell_i}$ and the Gaussian integer $\alpha_{{\ell_i}+N(\alpha_{\ell_i})}$ on $L$, $i.e.$, $\ell_i<\ell_{i+1}<1+N(\alpha_{\ell_i})$.
\end{enumerate}
This verifies Conjecture 2 on the line $L$ for all $1<n\leq\ell_{10^{10}}$.

If either conjecture is true, then it would follow  that there are infinitely many Gaussian primes on every Gaussian line. But proving either conjecture for even one Gaussian line (with $\ao\neq 0$) would give a Gaussian line with infinitely many Gaussian primes, and hence  a quadratic polynomial that takes on infinitely many rational prime values.

\section{Divisibility on Gaussian Lines.}\label{div}
Every second integer on the real line is divisible by 2, every third by 3, every fourth by 4, and so on. We wondered if this basic periodicity property of divisibility extends to Gaussian lines,  and furthermore, if there is a simple way to characterize those Gaussian primes that occur as divisors  on a particular Gaussian line.  
In this section we show that the answer to both of these questions is {\em YES}.

\medskip

Throughout this section, let $L$ be a primitive Gaussian line with $\alpha_0=a+bi$ and $\delta = c+di$ as defined in Section \ref{background}.   Then $\ao$ and $\delta$ are relatively prime Gaussian integers, $c$ and $d$ are relatively prime rational integers, and the Gaussian integers on $L$ are exactly the numbers $\alpha_n=\alpha_0+\delta n$, $n\in \Z$. 
Also, recall the function $\nu:\Z[i]\rightarrow \Z$ defined in (\ref{morm}) as it is used here and throughout the rest of the paper.
%

In the special case where $L$ is the real line, we have $\alpha_0=0$, $\delta=1$, and $\alpha_n=n$, $n\in\Z$. In this case,  divisibility of integers on the line $L$ by a rational integer $r$ is periodic with period $r$. Our first theorem shows that this periodicity generalizes to arbitrary primitive Gaussian lines, specifically that divisibility by a Gaussian integer $\beta$ is periodic with period $\nu(\beta)$. Note that the periodicity of divisibility on the real line is a special case of the following theorem.

\begin{theorem} \label{periodicity}
Suppose $\beta\in\Z[i]$ divides some Gaussian integer $\alpha_t$ on $L$.  
Then  $\beta$ divides $\alpha_{n}$ if and only if $n \equiv t \pmod {\nu(\beta)}$.
\end{theorem}

\begin{proof}
Suppose $\beta $ divides $\alpha_{t}$ for some $t$. 
Then $\beta $ and $\delta$ are relatively prime, since any common divisor would also divide $\alpha_0= \alpha_t - \delta t$, but $\delta$ and $\alpha_0$ are relatively prime. Thus, $\beta$ divides $\alpha_n$  if and only if $\beta$ divides $\alpha_n-\alpha_t$, which in turn holds if and only if $\beta$ divides $n-t$ since $\alpha_n- \alpha_t = \delta(n-t)$.
But $n-t \in \Z$, so $\beta$ divides $\alpha_n$ if and only if $\nu(\beta)$ divides $n-t$.
\end{proof}

Theorem \ref{periodicity} implies that consecutive Gaussian integers  $\alpha_n$ and $\alpha_{n+1}$ on $L$  are always relatively prime over $\Z[i]$, just as consecutive rational integers on the real line  are always relatively prime over $\Z$. Also, because Theorem \ref{periodicity} is about Gaussian integers that divide some element on $L$, a natural follow-up problem is to characterize those Gaussian integers that  occur as divisors of elements on $L$.  In this section we specialize to rational integer and Gaussian prime divisors, and in Section \ref{divisor} we give the complete characterization of the set of Gaussian integer divisors.

We define the {\em divisor set of $L$}, denoted $\DL$, to be the set of Gaussian integers that divide some Gaussian integer on $L$. Our main result in Section \ref{divisor} (Theorem \ref{bigtheorem}) is a complete characterization of this set.  Here we begin by characterizing two of its subsets, the {\em rational set} and the {\em Gaussian-prime set},  which we need for our work in Section~\ref{CRT}.  We define the {\em rational set of $L$}, denoted $\ZL$, to be the set of rational integers that divide some Gaussian integer on $L$, and the {\em Gaussian-prime set of $L$}, denoted $\PL$, 
to be the set of non-rational Gaussian primes that divide some Gaussian integer on $L$. For easy reference, below are the set theoretical  definitions of these three sets for a given Gaussian line $L$:
\begin{align*}
\ZL &=\{r \in \Z : r|\alpha_n \text{ for some } n \in \Z \}; \\
\PL &= \{ \pi\in\Z[i]  : \pi\ \text{is a Gaussian prime, }\pi \not\in \Z, \text{ and }\pi|\alpha_n \text{ for some } n \in \Z\};\\
\DL &= \{ \beta\in\Z[i]:\beta|\alpha_n\ \text{for some } n \in \Z \}.
\end{align*}
Note that an element in any of these three sets does not necessarily lie on the line $L$, but simply divides some Gaussian integer that lies on $L$. 

In general, the divisor set $\DL$ of $L$ is not closed under multiplication.  For example, suppose 
 $1+2i$ divides $\ao$ and $1-2i$ divides $\alpha_1$, so $1+2i, 1-2i\in\DL$. Since $\nu(1+2i)=\nu(1-2i)=5$, it follows from Theorem \ref{periodicity} that $1+2i$ and $1-2i$ both divide every fifth Gaussian integer on $L$, starting with $\ao$ and $\alpha_1$ respectively. Thus,
 no integer on $L$  is divisible by their product  $i.e.$,
  $(1+2i)( 1-2i)=5\notin\DL$, and $\DL$ is not closed under multiplication. Our first lemma shows that this type of restriction from Theorem~\ref{periodicity} is really the only property preventing the divisor set from being closed under multiplication.






\begin{lemma} \label{ncrt}
Let $\beta$ and $\gamma$ be in the divisor set $\DL$ of $L$. If $\nu (\beta)$ and $\nu (\gamma)$ are relatively prime, then $\beta \gamma$ is in $\DL$.
\end{lemma}

\begin{proof}
Suppose $\beta,\gamma\in\DL$. Then, by Theorem \ref{periodicity}, there exist integers $s$ and $t$ such that $\beta|\alpha_n$ if and only if $n \equiv s \pmod{\nu (\beta)}$ and $\gamma|\alpha_n$ if and only if $n \equiv t \pmod{\nu (\gamma)}$. By the Chinese Remainder Theorem, there is an $n$ that satisfies both congruences. Therefore, $\beta \gamma \in \DL$.
\end{proof}

We use Lemma \ref{ncrt} to prove our next theorem and characterize the rational set of $L$. Recall from (\ref{D})  that $\Delta=ad-bc$ is a rational integer associated to $L$.

\begin{theorem} \label{ZL} Let $r\in\Z$.  Then  $r$ is in the rational set $\ZL$ of $L$ if and only if r divides~$\Delta$.\end{theorem}

\begin{proof}  Note that if $r,s\in\Z$ satisfy $rs\in\ZL$, then $r\in\ZL$ and $s\in\ZL$ by the definition of the rational set.  
It follows from this and Lemma \ref{ncrt}    that it is sufficient to prove Theorem \ref{ZL}  for prime powers.

Let  $r=p^t$, where $p$ is a rational prime. Then $r\in\ZL$ if and only if $p^t$ divides $\alpha_n$ for some $n\in\Z$. We have that $\alpha_n=\ao+n\delta$, so $\re(\alpha_n)=a+nc$ and $\im(\al_n)=b+nd$. 
Thus, $p^t$ divides $\alpha_n$ if and only if $p^t$ divides both $a+nc$ and $b+nd$. Recall that $c$ and $d$ are relatively prime, so at least one of them is not divisible by $p$.  With out loss of generality, we assume that $p$ does not divide $c$.  Then $c$ has a multiplicative inverse modulo $p^t$. Thus we have:
 \beq p^t|\alpha_n &\Longleftrightarrow &a+nc\equiv 0\pmod{p^t}
 \quad  {\rm and}\quad  b+nd\equiv 0\pmod{p^t}\\
&\Longleftrightarrow  &b\equiv -nd\pmod {p^t}, \text{ where }  n\equiv -ac^{-1}\pmod {p^t}\\
&\Longleftrightarrow &b\equiv ac^{-1}d\pmod {p^t}\\ 
 &\Longleftrightarrow &ad\equiv bc\pmod{p^t}\\
 &\Longleftrightarrow &p^t|\,\Delta,
\eeq
as needed.
\end{proof}

Thus, the rational integers that divide some Gaussian integer $\alpha_n$ on $L$ are exactly the divisors of 
$\Delta$. Consequently, the rational set $\ZL$ of $L$ is finite unless $\Delta=0$; that is, unless $L$ is  the real or imaginary line. 
 Our next theorem characterizes the Gaussian prime set of $L$ and shows, by contrast, that this set is always infinite. 


 \begin{theorem}  \label{PL}  Let $\pi$ be a Gaussian prime with $\pi\not\in\Z$. Then $\pi\in\PL$ if and only if $\pi$ does not divide $\delta$. 
\end{theorem}

\begin{proof} 
First suppose $\pi$ divides $\delta$. Then $\pi$ does not divide $\alpha_n=\ao+\delta n$ for all $n$ since $\alpha_0$ and $\delta$ are relatively prime.   Thus, $\pi \not \in \PL$ in this case.

Conversely, suppose  $\pi$ does not divide $\delta$.  Let $\pi$ lie over the rational prime $p$.  If $p$ divides $\Delta$, then $p$ divides some Gaussian integer $\alpha_n$ on  $L$ by Theorem~\ref{ZL}. Thus $\pi$ also divides $\alpha_n$, and $\pi\in\PL$ as needed. 
Thus, from now on we assume $p$ does not divide $\Delta$, and show that $\pi\in\PL$ in this case as well.

As in (\ref{N}),  the norm of an arbitrary Gaussian integer $\alpha_n$ on $L$ can be viewed as a quadratic polynomial 
$$f(n)=N(\ao+\delta n)=N(\delta)n^2+Tr(\alpha_0\overline\delta\,)n+N(\alpha_0),$$
with discriminant  Disc$(f)=-4\Delta^2$.
If $p\neq 2$, then $p\equiv 1\pmod 4$ since $\pi\notin\Z$. In this case, $-1$ is a square modulo $p$ and so Disc$(f)$ is a non-zero square modulo $p$.  Therefore, $f(n)$ has two distinct roots modulo $p$, so there are $r,s\in\Z$, $r\not\equiv s\pmod p$, such that 
$N(\alpha_r)\equiv N(\alpha_s)\equiv 0\pmod p$.  It follows from Theorem \ref{periodicity} that $\pi$ and $\overline\pi$ both divide exactly one of $\alpha_r$ and $\alpha_s$.  Thus $\pi\in\PL$ in this case.
If $p=2$, then Disc$(f)\equiv 0\pmod p$ and $f$ has a double root modulo $p$.
It follows that  $\pi$ divides either $\alpha_0$ or $\alpha_1$.   Thus, $\pi\in\PL$ in this case as well.
\end{proof}

Since  $\delta\neq 0$, it follows from Theorem \ref{PL} that the divisor set of a Gaussian line always contains infinitely many Gaussian primes. In particular, we have the following corollary to Theorem \ref{PL}.

\begin{corollary} \label{minprimes} The divisor set $\DL$ of $L$ contains at least one Gaussian prime that lies over $p$ for every rational prime $p\equiv 1\pmod 4$.
\end{corollary}

\begin{proof}
Let $\pi$ be a Gaussian prime that lies over the rational prime $p\equiv 1\pmod 4$. Suppose that neither $\pi$ nor $\overline{\pi}$ are in $\DL$. Then neither is in $\PL$ and so both divide $\delta$  by Theorem \ref{PL}. Thus $p$ divides $\delta$ and $p$ is a common divisor of $c$ and $d$, which contradicts $L$ being primitive.\end{proof}

Taken together, Theorems \ref{ZL} and \ref{PL} imply that if a Gaussian prime $\pi$ divides $\delta$, and $\pi$ lies over $p$, then $p$ does not divide $\Delta$ (or, equivalently, $\pi$ does not divide $\Delta$).  This can also be seen directly:  If $\pi$ is a common divisor of both $\delta$ and $\Delta$, then $\pi$ divides $d$ since $d\ao=\Delta+b\delta$ and $\ao$ and $\delta$ are relatively prime.  Now, $\delta=c+di$, so $\pi$ also divides $c$.  But $c, d\in \Z$, so it follows that $p$ is a common divisor of $c$ and $d$, which contradicts $L$ being primitive.

Theorems \ref{ZL} and \ref{PL} characterize the rational and Gaussian prime set of a Gaussian line.  In Section \ref{divisor} we use these theorems to give a complete characterization of the divisor set as well.  First we turn to some consequences of the theorems in this section.

\section{The Chinese Remainder Theorem for Gaussian Lines.}\label{CRT}

In this section we prove a theorem about Gaussian lines that is analogous to the Chinese Remainder Theorem for $\Z$.  We also use the Chinese Remainder Theorem for $\Z[i]$ to prove that  there are always infinitely many Gaussian lines that satisfy any given CRT-type divisibility properties. 

\medskip

The Chinese Remainder Theorem (CRT) for $\Z$  implies that there will always be a solution to a system of linear congruences over $\Z$  when the moduli are pairwise relatively prime. It is well known that this theorem generalizes with the same proof to the Gaussian integers (or to any Euclidean domain).  We state this more general version here since we will need it in our later work.
%
%

\begin{theorem}[CRT for $\zi$] \label{chinese}
Let $\mu_1,\mu_2,\ldots, \mu_k$ be pairwise relatively prime Gaussian integers and $\beta_1, \beta_2,\ldots, \beta_k$ be arbitrary Gaussian integers.
Then the system of $k$ congruences
$$x \equiv \beta_j \pmod{\mu_j}, \ 1\leq j\leq k,$$
has a unique solution $\tau\in\Z[i]$ modulo $\mu_1\mu_2\ldots \mu_k$.
\end{theorem}

Note that CRT for $\Z$ is just Theorem \ref{chinese} with $\beta_j, \mu_j\in\Z$, $1\leq j\leq k$. In the spirit of this paper, we extend CRT for $\Z$ to CRT for  Gaussian lines. First we restate CRT for $\Z$ in terms of divisibility since the analogous statement for Gaussian lines is given in terms of divisibility.

\begin{theorem}[CRT for $\Z$] \label{crtZ}
Let $m_1, m_2,\ldots, m_k$ be pairwise relatively prime rational integers and $b_1, b_2,\ldots, b_k$ be arbitrary rational integers.  Then there is a unique rational integer $t$ modulo $m_1 m_2\cdots m_k$ such that
$$m_1|{(t+b_1)}, \ m_2|{(t+b_2)}, \ \ldots, \ m_k|{(t+b_k)}.$$
\end{theorem}

We use the function $\nu:\Z[i]\rightarrow\Z$ defined (\ref{morm}) to extend Theorem \ref{crtZ} to any Gaussian line.  Since $\nu(n)=n$ for all $n\in\Z$,  the following theorem is exactly CRT for $\Z$ when $L$ is the real line.

\begin{theorem}[CRT for Gaussian lines] \label{crtGL}  Let $L$ be a primitive Gaussian line, and suppose
$\mu_1, \mu_2,\ldots, \mu_k$ are Gaussian integers in the divisor set $\DL$ of $L$ such  that $\nu(\mu_1),\nu(\mu_2), \ldots, \nu(\mu_k)$ are pairwise relatively prime. Let $b_1, b_2,\ldots, b_k$ be arbitrary rational integers. Then there is a unique rational integer $t$ modulo $\nu(\mu_1)\nu(\mu_2) \cdots \nu(\mu_k)$ such that
$$ \mu_1|\alpha_{t+b_1}, \ \mu_2|\alpha_{t+b_2}, \ \ldots, \ \mu_k|\alpha_{t+b_k}.$$
\end{theorem}

\begin{proof}
Since $\mu_j\in\DL$, $1\leq j\leq k$, it follows from  Theorem \ref{periodicity} that for each $j$  there exists $m_j\in\Z$ such that $\mu_j$ divides the Gaussian integer $\alpha_{n}$ on $L$  if and only if $n \equiv m_j \pmod {\nu(\mu_j)}$.
By Theorem \ref{crtZ}, the  system of $k$ congruences
$$x \equiv m_j-b_j \pmod {\nu(\mu_{j})}, \  1\leq j\leq k,$$
has a unique solution $x\equiv t\pmod{\nu(\mu_1)\nu(\mu_2) \cdots \nu(\mu_k)}$.
Thus, for $1\leq j\leq k$, we have $t+b_j\equiv m_j \pmod{\nu(\mu_j)}$  and  $\alpha_{t+b_j}$ is divisible by $\mu_j$ as needed.
\end{proof}

Now, suppose you want to find a primitive Gaussian line that satisfies certain CRT-type divisibility properties. For instance, suppose you want a line where $2+i$ divides $\alpha_1$, $2+3i$ divides $\alpha_2$, and $4080 + 1397i$ divides $\alpha_3$. It follows from our next theorem that infinitely many such lines exist (one example in this case is the line defined by $\alpha_0=1$ and $\delta=6297+8234i$), and the proof shows how to construct them.

\begin{theorem} \label{newchinese}
Let $b_1, b_2, \ldots, b_k $ be 
rational integers and $\mu_1, \mu_2, \ldots, \mu_k$ be pairwise relatively prime Gaussian integers.  Then there are infinitely many primitive Gaussian lines  $L$ such that $\mu_j $ divides the Gaussian integer $ \alpha_{b_j}$ \hskip-.03in on $L$ for all $1 \leq j\leq k$.
\end{theorem}

\begin{proof}  
To show there are infinitely many primitive Gaussian lines $L$ that satisfy the desired divisibility conditions, we show that there are infinitely many Gaussian integers
 $\alpha_0=a+bi$ and $\delta=c+di$ that satisfy all of the following properties:
\begin{quote} \begin{enumerate}
 \item[(a)] $N(\ao+n\delta)>N(\ao)$ for all $n\neq 0$, $n\in\Z$;
 \item[(b)] gcd$(c,d)=1$ and $c\geq 0$;
  \item[(c)] $\alpha_0$ and $\delta$ are relatively prime over $\zi$;
  \item[(d)] $\mu_j $ divides  $ \alpha_{b_j}\hskip-.035in=\ao+b_j\delta$  for all $1 \leq j\leq k$.
 \end{enumerate}
\end{quote}

We first choose $\ao$. For $1 \leq j\leq k$, let $\gamma_j\in\Z[i]$ be a common divisor of $\mu_j$ and $b_j$ with maximal norm (each $\gamma_j$ is uniquely defined up to multiplication by a unit in $\zi$). Let $\lambda$ be any Gaussian integer that is relatively prime to both $\mu_1\mu_2\cdots\mu_k$ and $b_1b_2\cdots b_k$.
Define $\ao$ by 
$$\alpha_0=\lambda\prod_{j=1}^k \gamma_j\in\zi.$$
There are infinitely many possibilities for $\lambda$, so there are infinitely many possibilities for $\ao$.

For each $\ao$, we show there are infinitely many  $\delta=c+di\in\zi$ such that the above properties (a)--(d) are satisfied.  Property (d) is equivalent to $\delta$ being a solution to the system of $k$ congruences
$$\ao+b_j x\equiv 0\pmod{\mu_j}, \ \ 1 \leq j\leq k.$$
Dividing by $\gamma_j$ for each $j$, this is equivalent to $\delta$ being a solution to the system
$$x\equiv -\left({\ao\over\gamma_j}\right)\kappa_j^{-1} \pmod{\omega_j}, \ \ 1 \leq j\leq k,$$
where each $\kappa_j=b_j/\gamma_j\in\zi$ is relatively prime to $\omega_j=\mu_j/\gamma_j\in\zi.$
Note that each $\ao/\gamma_j$ is also relatively prime to $\omega_j$  since 
$\omega_1,\omega_2,\ldots,\omega_k$
 are pairwise relatively prime. Thus, any solution to this latter system of congruences is relatively prime to the product $\omega_1\omega_2\cdots\omega_k$. 
Since $\delta$ will be a solution, we include an additional congruence to insure that any solution is also relatively prime to $\alpha_0$ and so property~(c) will automatically be satisfied.  Let $\beta$ be the product of all the Gaussian primes that divide $\alpha_0$ but do not divide 
 $\omega_1\omega_2\cdots\omega_k$, and  let $\beta=1$ if no such Gaussian primes exist.  
 Then $\delta$ is relatively prime to $\ao$ if it is relatively prime to both $\beta$ and 
 $\omega_1\omega_2\cdots\omega_k$.
 Thus,  to insure that properties (c) and (d) are both satisfied,  it is sufficient that $\delta$  be a solution to the following system of $k+1$ congruences:
\begin{align*}
x & \equiv 1\pmod\beta,\  \ {\text{and}} \\
x &\equiv -\left({\ao\over\gamma_j}\right)\kappa_j^{-1} \pmod{\omega_j}, \ \ 1 \leq j\leq k. 
\end{align*}
This system has a unique solution 
$\tau=r+si\in\zi$ modulo $\beta\omega_1\omega_2\cdots\omega_k$ by CRT for Gaussian integers (Theorem
\ref{chinese}) since $\beta,\omega_1,\omega_2,\ldots,\omega_k$
 are pairwise relatively prime.  Thus, it remains to construct $\delta=c+di$ that satisfies  properties (a) and (b), and such that $\delta\equiv\tau\pmod{\beta\omega_1\omega_2\cdots\omega_k}$, so that properties (c) and (d) hold as well.  
 
To satisfy property (a), we construct $\delta=c+di$ 
 such that 
 $$N(\alpha_n)=N(\ao+n\delta)=(c^2+d^2)n^2+2(ac+bd)n+a^2+b^2, \ \ n\in\Z,$$
obtains its minimum value only when $n=0$. For any $c,d\in\Z$, the quadratic function,
$$f(x)=(c^2+d^2)x^2+2(ac+bd)x+a^2+b^2, \ \ x\in\R,$$
obtains its absolute minimum when $f^\prime(x)=0$, $i.e.$, when 
$x={-(ac+bd)/({c^2+d^2})}.$
Thus, since $f$ is symmetric, for property (a) to be satisfied and $f(0)$ to be the minimum integer value of $f$, it is sufficient that $c$ and $d$ satisfy
\begin{equation}\label{1prime}-{1\over 2}<{{ac+bd}\over{c^2+d^2}}<{1\over 2}.\end{equation}
 For a fixed $d$,
$$\lim_{c\rightarrow\infty}\left({{ac+bd}\over{c^2+d^2}}\right)=0,$$
so (\ref{1prime})  holds for all $c$ larger than some bound that depends on $d$.  We use this fact to complete the proof.

It is sufficient to choose $\delta=c+di$ such that  (\ref{1prime}) holds,  gcd$(c,d)=1$, $c\geq 0$, and 
$\delta\equiv\tau\equiv r+si\pmod{\beta\omega_1\omega_2\cdots\omega_k}$.
Let $M=N(\beta\omega_1\omega_2\cdots\omega_k)\in\Z$. We first consider $s=0$.  In this case, $\tau=r$ is relatively prime to $M$ since is a non-zero rational integer that is relatively prime to 
$\beta\omega_1\omega_2\cdots\omega_k$.  
It follows from Dirichlet's Theorem on Primes in Arithmetic Progressions, that there are infinitely many rational primes congruent to $r$ modulo $M$. Thus, we can choose  a rational  prime $p$ such that $p\equiv r\pmod{M}$, $p>M$, and $p$ is large enough so that (\ref{1prime}) holds for $c=p$ and $d=M$.
 Define $\delta$ by 
$\delta = p+Mi$. Then, (\ref{1prime}) holds and gcd$(c,d)=1$, since $p$ is prime and larger than  $M$. 
Also, $\delta\equiv\tau\pmod{M}$, so $\delta\equiv\tau\pmod{\beta\omega_1\omega_2\cdots\omega_k}$, since $\beta\omega_1\omega_2\cdots\omega_k$ divides $M$.
Thus, $\ao$ and $\delta$ define a primitive Gaussian line that satisfies the divisibility conditions stated in Theorem~\ref{newchinese}.  Moreover,  according to Dirichlet's Theorem, there are infinitely many choices of the prime $p$. Thus, there are infinitely many choices for $\delta$, and so infinitely many primitive Gaussian lines with the same $\alpha_0$ that satisfy the conditions.


Similarly, if $r=0$, then $\tau=si$, where $s$ is a non-zero rational integer that is relatively prime to $M$.  Proceed as above to get a rational prime $p\equiv s\pmod{M}$, $p>M$, and  $p$ is large enough so that (\ref{1prime}) holds for $c=M$ and $d=p$.
 Define $\delta$ by $\delta=M+pi$. Then $\ao$ and $\delta$ define a primitive Gaussian line that satisfies the divisibility conditions stated in Theorem \ref{newchinese}. Again, by Dirichlet's Theorem, there are infinitely many choices of the prime $p$ and thus infinitely  many primitive Gaussian lines with this $\alpha_0$ that satisfy these  conditions.

Finally, suppose $r$ and $s$ are both non-zero rational integers.  Let $h$ be the smallest positive rational divisor of $r$ such that gcd$(r/h,M)=1$.  Again by 
Dirichlet's Theorem, we can find a  rational prime $p>s$ such that $p\equiv r/h\pmod M$ 
and $p$ is large enough so that (\ref{1prime}) holds for $c=ph$ and $d=s$.
Define $\delta$ by 
$$\delta=ph+si.$$ 
Then $\delta\equiv\tau\pmod{\beta\omega_1\omega_2\cdots\omega_k}$.
To see that gcd$(ph,s)=1$, first observe that gcd$(p,s)=1$ since $p>s$ is prime. Also, gcd$(h,s)=1$, since any common rational prime divisor $q$ of $h$ and $s$ is also a common divisor of $\tau$ and $M$.  Hence, there is a Gaussian prime that lies over $q$ that divides both $\tau$ and $\beta\omega_1\omega_2\cdots\omega_k$, which is a contradiction since they are relatively prime.   Thus, as above, $\ao$ and $\delta$ define a primitive Gaussian line that satisfies the required divisibility conditions, and again there are infinitely many choices of $\delta$ by Dirichlet's Theorem.

\end{proof}

\section{The Divisor Set of a Gaussian Line.}\label{divisor}

We now return to questions about divisibility on Gaussian lines related to those discussed in Section \ref{div}.  For a given Gaussian line $L$, we first characterize those Gaussian-prime powers that exactly divide some Gaussian integer on $L$.  Using this, our main theorem in this section gives a complete characterization of the divisor set $\DL $ of $L$.

\medskip

Theorem  \ref{PL} in Section \ref{div} resolves the question of which Gaussian primes occur in the divisor set $\DL$ of $L$, but does not address division by prime powers.  For example, Theorem~\ref{PL}  does not answer the following  question: If $\pi\in\DL$, then is $\pi^{100}$ guaranteed to be in $\DL$? Nor does it say anything about which prime powers $\pi^k$ {\em exactly divide} some Gaussian integer $\alpha_n$ on $L$ ($i.e.$, $\pi^k$ divides $\alpha_n$, but $\pi^{k+1}$ does not). For example,  if $\pi^{50}\in\DL$, then certainly $\pi, \pi^2,\ldots,\pi^{49}\in\DL$, but is $\pi$ guaranteed to exactly divide some Gaussian integer  on $L$? What about $\pi^2$ or $\pi^3$ or any other power of $\pi$?  Our next theorem shows that the answer to all of these questions is {\em YES} whenever $\pi$ lies over a rational prime $p\equiv 1\pmod 4$, but is conditional for other values of $p$.  We restrict to lines with $\Delta\neq 0$ since this simplifies the proof and exact division by all prime powers holds on the real and imaginary lines. 

\begin{theorem} \label{exact}  Let  $L$ be a primitive Gaussian line with $\Delta\neq 0$. Suppose $\pi$ is a Gaussian prime that lies over the rational prime $p$.
\begin{enumerate}
\item If $p\equiv 1\pmod 4$, then the following are equivalent:
\begin{enumerate}
\item $\pi$ does not divide $\delta$.
\item $\pi^k\in\DL$ for some positive integer $k$.
\item  For every positive integers $r$,  $\pi^r$ exactly divides some Gaussian integer  on $L$. In particular, $\pi^r\in\DL$ for all positive integers $r$.
\end{enumerate}
\item If $p=2$,  then the following are equivalent:
\begin{enumerate}
\item $1+i$ does not divide $\delta$.
\item  $(1+i)^k\in\DL$ for some positive integer $k$.
\item  Let $2^s$ be the exact power of $2$ that divides $\Delta$, and $\beta\in\zi$ have 2-power norm.  Then $\beta$ exactly divides some Gaussian integer $\alpha_n$ on $L$ if and only if $\beta$ is an associate of $2, 2^2, \ldots, 2^s$, or $2^{s}(1+i)$.  That is, $(1+i)^t\in\DL$ if and only if $0\leq t\leq 2s+1$, but  $(1+i)^t$ exactly divides a Gaussian integer on $L$ if and only if in addition $t$ is even or $t=2s+1$.
\end{enumerate}
\item If $p\equiv 3\pmod 4$ (so $\pi$ is an associate of $p$), then $p^k$ exactly divides some Gaussian integer $\alpha_n$ on $L$ if and only if $p^k$ divides $\Delta$.
\end{enumerate}
\end{theorem}

\begin{proof} We consider the three cases separately.

\medskip

\noindent {\underline{\em{Case 1}}:}

Suppose $p\equiv 1\pmod 4$.  Statements 1(a) and 1(b) are equivalent by Theorem \ref{PL}.  Since 1(c) trivially implies 1(b), we only need to show that 1(b) implies 1(c).  For this, suppose that $\pi^k\in\DL$, say $\pi^k$  divides $\alpha_m$. Then $\pi^h$ exactly divides $\alpha_m$ for some $h\geq k$.   Let $r$ be a positive integer.   If $r< h$, then  1(c) holds since $\pi^r$ exactly divides $\alpha_n$ for $n=m+p^rq$, where $q$ is any integer not divisible by $p$.  To see this, write $\alpha_n=\alpha_0+(m+p^rq)\delta=\alpha_m+p^rq\delta$, and use that $\pi^h$ exactly divides $\alpha_m$ while $\pi^r$ exactly divides $p^rq\delta$.  Note that by considering the special case where $r=1$, this shows in general that if a Gaussian prime $\pi$ does not divide $\delta$ then $\pi$ {\em exactly} divides some Gaussian integer $\alpha_n$ on $L$.

We use induction and the general fact for $r=1$ given above to show that 1(c) holds for $r\geq h$ as well.  If $r=h$, then $\pi^r$ exactly divides $\alpha_m$ by hypothesis, so 1(c) holds in this case.  Suppose it holds for some $t\geq h$, say $\pi^t$ exactly divides $\alpha_s$. Let $\omega=\alpha_s/\pi^t\in\zi$.
 For $q\in\Z$, consider 
$$\alpha_{s+p^{t}q}=\alpha_s+p^{t}q\delta=\pi^t(\omega+\overline\pi^{t}\delta q),$$
where $p=\pi\overline\pi$ and $\overline\pi$ is not an associate of $\pi$ since $p\equiv 1\pmod 4$. 
Now, $\overline{\pi}^{t}\delta$ has no rational integer divisors since $\pi\nmid\delta$.  Also, $\omega$ and $\overline{\pi}^{t}\delta$ are relatively prime since $\alpha_s$ and $p^t\delta$ are relatively prime. Thus, the numbers $\omega+\overline\pi^{t}\delta q$, $q\in\Z$, are the Gaussian integers on a different primitive Gaussian line $L^\prime$ with $\delta^\prime=\overline\pi^{t}\delta$. Since $\pi \nmid \delta^\prime$, it follows from the general result  for $r=1$, that there is a $q_0\in\Z$ such that $\pi$ exactly divides the Gaussian integer $\omega+\overline\pi^{t}\delta q_0$ on $L^\prime$. Thus  $\pi^{t+1}$ exactly divides the Gaussian integer $\alpha_n$ on $L$ for $n=s+p^{t}q_0$, and 1(c) holds for $r=t+1$.  By induction it holds for all $r$.


\medskip

\noindent {\underline{\em{Case 2}}:}

Suppose  $p=2$. As above, it is sufficient to prove statement 2(b) implies 2(c). Suppose $(1+i)^k\in\DL$ for some positive integer $k$.  Then $(1+i)\in\DL$ and $1+i$ does not divide $\delta$. 
Let $2^s$, $s\geq 0$, be the exact power of $2$ that divides $\Delta$.  Then  $2^s\in\DL$, but $2^{s+1}\not\in\DL$ by Theorem \ref{ZL}. Since $2$ ramifies in $\zi$, this is equivalent to $(1+i)^{2s}\in\DL$, but $(1+i)^{2s+2}\not\in\DL$. 

We first claim that $(1+i)^{2s+1}\in\DL$. For this, note that since  $(1+i)^{2s}\in\DL$, there is a  Gaussian integer $\alpha_m$ on $L$ such that $(1+i)^{2s}$ divides $\alpha_m$.  If $(1+i)^{2s+1}$ divides $\alpha_m$ then $(1+i)^{2s+1}\in\DL$ as claimed.  So suppose $(1+i)^{2s+1}$ does not divide $\alpha_m$. By Theorem \ref{periodicity}, $(1+i)^{2s}$ divides $\alpha_{m+2^s}$ since $\nu((1+i)^{2s})=2^s$. Now, 
$$\alpha_{m+2^s}=\alpha_m+2^s\delta=2^s(\omega+\delta),$$
where $\omega=\alpha_m/2^s\in\zi$ is not divisible by $1+i$.  Since neither $\omega$ nor $\delta$ is divisible by $1+i$, their sum must be divisible by $1+i$. Thus, $(1+i)^{2s+1}$ divides $\alpha_{m+2^s}$, and $(1+i)^{2s+1}\in\DL$ in this case as well.

Thus we have  $(1+i)^t\in\DL$ if and only if $0\leq t\leq 2s+1$, and so it remains to consider exact division by  $(1+i)^t$. We consider $t$ even and $t$ odd separately.  First suppose that  $t=2h$ is even.  We claim that  $(1+i)^t$ exactly divides some Gaussian integer $\alpha_n$ on $L$, or equivalently, that $2^h$ divides  $\alpha_n$ but  $2^h(1+i)$ does not.  This is true when $t=s$ since $(1+i)^{2s}$  exactly divides either $\alpha_m$ or $\alpha_{m+2^s}$ by the preceding paragraph. So suppose that for some $h$, $0<h\leq s$, we have $(1+i)^{2h}$ exactly divides $\alpha_n$ for some $n$.  Consider,
$$\alpha_{n+2^{h-1}}=\alpha_n+2^{h-1}\delta=2^{h-1}(\omega+\delta),$$
where $\omega=\alpha_n/2^{h-1}\in\zi$ is divisible by $(1+i)^2$ and $\delta$ is not divisible by $1+i$. Thus, $\omega+\delta$ is not divisible by $1+i$ , and so $(1+i)^{2h-2}$ exactly divides $\alpha_{n+2^{h-1}}$.
The claim for odd $t$ follows by induction.

Now suppose $t$ is odd and $(1+i)^t$ exactly divides some Gaussian integer $\alpha_r$ on $L$.  For instance, this holds for $t=2s+1$ since $(1+i)^{2s+1}$  exactly divides either $\alpha_m$ or $\alpha_{m+2^s}$.  Write $t=2j+1$, so $\nu((1+i)^t)=2^{j+1}$. Thus, by Theorem \ref{periodicity}, $(1+i)^t$ divides $\alpha_n$ for $n=r+2^{j+1}q$, $q\in\Z$. Now,
$$\alpha_n=\alpha_{r+2^{j+1}q}=\alpha_r+2^{j+1}\delta q=(1+i)^t\left( \omega+\mu(1+i)\delta q\right),$$
where $\omega=\alpha_r/(1+i)^t\in\zi$ is not divisible by $1+i$ and $\mu\in\zi$ is a unit.  Now, the real and imaginary parts of $\mu(1+i)\delta$ must be relatively prime since $1+i$ does not divide $\delta$ and the real and imaginary part of $\delta$ are relatively prime.  Also, $\omega$ and  $\mu(1+i)\delta$ are relatively prime over $\zi$ since $1+i$ does not divide $\omega$ and $\alpha_r$ and $\delta$ are relatively prime.  
Thus, the numbers $\omega+(1+i)\delta q$, $q\in\Z$, are the Gaussian integers on a different primitive Gaussian line $L^\prime$ with $\delta^\prime=(1+i)\delta$. Since $1+i$ divides $\delta^\prime$, it follows from 
Theorem \ref{ZL} that none of the Gaussian integers $\omega+(1+i)\delta q$ are divisible by $1+i$, that is, 
$(1+i)\not\in{{\mathcal{D}}(L^\prime)}$.  Thus, $(1+i)^{t+1}\not\in\DL$, or equivalently, $2^{j+1}\not\in\DL$.  This is a contradiction unless $j=s$. Therefore, if $t$ is odd then $(1+i)^t$ exactly divides some Gaussian integer on $L$ if and only if $t=2s+1$.

%
%
%

\medskip

\noindent {\underline{\em{Case 3}}:}

Suppose $p\equiv 3\pmod 4$. Then $p$ remains prime in $\zi$ and $\pi$ is an associate of $p$.  By Theorem \ref{ZL}, we know that then $p^k$  divides some Gaussian integer $\alpha_n$ on $L$ if and only if $p^k$ divides $\Delta$. For {\em exact} divisibility, let $s$ be such that $p^s$ exactly divides $\Delta$.  Then $p^s$ exactly divides some $\alpha_m$ on $L$ since $p^{{s+1}}\not\in\DL$.  Then, as in the case $p=2$, we have that $p^{s-1}$ exactly divides
$$\alpha_{m+p^{s-1}}=\alpha_m+p^{s-1}\delta=p^{s-1}\left( \omega+\delta \right),$$
since $\omega=\alpha_m/p^{s-1}$ is divisible by $p$ but $\delta$ is not.
Continue in the same way to get that $p^k$ exactly divides some Gaussian integer on $L$ for all $k$ with $0\leq k\leq s$.
\end {proof}

Putting Theorem \ref{exact}  together with the results in Section \ref{div} yields a characterization of the divisor set $\DL$ of $L$ as follows.

\begin{theorem}\label{bigtheorem} Let  $L$ be a primitive Gaussian line with $\Delta\neq 0$.  A Gaussian integer $\beta$ is in the divisor set $\DL$ of $L$ 
if and only if $\beta$ can be written as
$$\beta= \mu r(1+i)^t\pi_1^{k_1}\pi_2^{k_2} \cdots \pi_m^{k_m},$$
where the variables in this expression are defined as follows: 
\begin{enumerate}
\item[(a)]  $\mu\in\{\pm1, \pm i\}$ is a unit in $\zi$;
\item[(b)]  $r$ is a rational integer that divides $\Delta$;
\item[(c)] $t=0$ if $1+i$ divides $\delta$, and $t\in \{0, 1\}$ otherwise;
\item[(d)] For $1\leq j\leq m$, $\pi_j$ is a Gaussian prime such that  $\pi_j$ does not divide $\delta$,
$N(\pi_j)\neq 2$, and $N(\pi_j)\neq N(\pi_n)$ for $j\neq n$;
\item[(e)] For $1\leq j\leq m$, $k_j \geq 0$ is a rational integer.
\end{enumerate}
\end{theorem}

\begin{proof}
By Lemma \ref{ncrt}, it is sufficient to characterize those $\beta\in\DL$ where  $\nu(\beta)$ is a prime power. Thus,  let $p$ be a rational prime and $\beta\in\zi$ satisfy $\nu(\beta)=p^n$ for some positive integer $n$.  

First suppose $p\equiv 1\pmod 4$, and let $\pi$ be a Gaussian primes that lies over $p$.  We may assume $\pi\in\PL$ by Corollary \ref{minprimes} of Theorem \ref{PL}. If $\overline\pi\not\in\PL$, then by Theorems \ref{ZL} and~\ref{exact}, $\beta\in\DL$ if and only if $\beta = \mu p^t\pi^k$, where  $\mu$ is a unit in $\zi$ and   $t$ and $k$ are non-negative integers, and $p^t$ divides $\Delta$.  If, in addition, $\overline\pi\in\PL$, then $\beta$ can also be of the form $\mu p^t\overline\pi^k$.

If $p=2$ then, up to associates, $1+i$  is the only Gaussian prime that lies over $p$.
Let $2^s$ be the power of $2$ that exactly divides $\Delta$.  It follows from Theorem \ref{exact} that $\beta\in\DL$ if and only if $\beta = \mu 2^r(1+i)^t$, where  $\mu$ is a unit in $\zi$,  $0\leq r\leq s$,  and
$t=0$ if $1+i$ divides $\delta$ and $t\in \{0, 1\}$ otherwise.

Finally, if $p\equiv 3\pmod 4$, then $p$ remains prime in $\zi$.  In this case, it follows from Theorem \ref{ZL} that $\beta\in\DL$ if and only if $\beta=\mu p^r$, where $\mu$ is a unit in $\zi$ and $p^r$ divides $\Delta$.
\end{proof}

\end{document}